\documentclass{amsart}

\usepackage[utf8]{inputenc}
\usepackage{verbatim}
\usepackage{graphicx}
\usepackage{graphicx,caption2,psfrag,float,color}
\usepackage{amssymb}
\usepackage{amscd}
\usepackage{amsmath}
\usepackage{mathrsfs}

\newtheorem{theorem}{Theorem}[section]

\newtheorem{cor}[theorem]{Corollary}
\newtheorem{lemma}[theorem]{Lemma}

\numberwithin{equation}{subsection}
\newtheorem{definition}[theorem]{Definition}

\title{Estimates of sums related to the Nyman-Beurling criterion for the Riemann Hypothesis}
\author{Helmut Maier and Michael Th. Rassias}
\date{\today}
\address{Department of Mathematics, University of Ulm, Helmholtzstrasse 18, 89081 Ulm, Germany.}
\email{helmut.maier@uni-ulm.de}
\address{Institute of Mathematics, University of Zurich, CH-8057, Zurich, Switzerland
 \& Institute for Advanced Study, Program in Interdisciplinary Studies,
1 Einstein Dr, Princeton, NJ 08540, USA.}
\email{michail.rassias@math.uzh.ch, michailrassias@math.princeton.edu}\thanks{}

\begin{document}

\maketitle
 
\begin{abstract} 
We give an estimate for sums appearing in the Nyman-Beurling criterion for the
Riemann Hypothesis containing the M\"obius function. The estimate is remarkably
sharp in comparison to estimates of other sums containing the M\"obius function.
The methods intensively use tools from the theory of continued fractions and 
from the theory of Fourier series.

\newline

\end{abstract}
\section{Introduction and statement of result}
According to the approach of Nyman-Beurling-B\'aez-Duarte (see \cite{baez1}, \cite{BEC}) to the Riemann Hypothesis, the Riemann Hypothesis is true if and only if 
$$\lim_{N\rightarrow \infty} d_N^2=0\:,$$
where
\[
d_N^2=\inf_{D_N}\frac{1}{2\pi}\int_{-\infty}^\infty\left|1-\zeta D_N\left(\frac{1}{2}+it\right)\right|^2\frac{dt}{\frac{1}{4}+t^2}\tag{1.1}
\]
and the infimum is over all Dirichlet polynomials 
$$D_N(s):=\sum_{n=1}^{N}\frac{a_n}{n^s}\:,\ a_n\in\mathbb{C}\:,$$
of length $N$ (see \cite{bcf}).\\
Various authors (cf. \cite{baez1}, \cite{baez2}, \cite{burnol}) have investigated the question, which 
asymptotics hold for $d_N$ if the Riemann Hypothesis is true.\\
In \cite{bcf} the following was shown:\\
If the Riemann Hypothesis is true and if 
$$\sum_{|Im(\rho)|\leq T}\frac{1}{|\zeta'(\rho)|^2}\ll T^{\frac{3}{2}-\delta}$$
for some $\delta>0$, then
$$\frac{1}{2\pi}\int_{-\infty}^{\infty}\left|1-\zeta V_N\left(\frac{1}{2}+it\right)\right|^2\frac{dt}{\frac{1}{4}+t^2}\sim\frac{2+\gamma-\log4\pi}{\log N}$$
for 
\[
V_N(s):=\sum_{n=1}^N\left(1-\frac{\log n}{\log N}\right)\frac{\mu(n)}{n^s}\:.\tag{1.2}
\]
Moreover, it follows from work of Burnol \cite{burnol} that among all Dirichlet polynomials $D_N(s)$ the infimum in (1.1) is assumed for $D_N(s)=V_N(s)$. It thus is of interest to obtain an unconditional estimate for the integral in (1.1).\\
If we expand the square in (1.1) we obtain
\begin{align*}
d_N^2&=\inf_{D_N}\Bigg(\int_{-\infty}^\infty\left(1-\zeta\left(\frac{1}{2}+it\right)D_N\left(\frac{1}{2}+it\right)-\zeta\left(\frac{1}{2}-it\right)D_N\left(\frac{1}{2}-it\right)\right)\frac{dt}{\frac{1}{4}+t^2}\\
&\ \ \ +\int_{-\infty}^\infty\left|\zeta\left(\frac{1}{2}+it\right)\right|^2\left|D_N\left(\frac{1}{2}+it\right)\right|^2\frac{dt}{\frac{1}{4}+t^2}\Bigg)\:.
\end{align*}
The last integral evaluates as 
$$\sum_{1\leq h,k\leq N} a_h\bar{a_k} h^{-1/2}k^{-1/2} \int_{-\infty}^\infty\left|\zeta\left(\frac{1}{2}+it\right)\right|^2\left(\frac{h}{k}\right)^{it}\frac{dt}{\frac{1}{4}+t^2}\:.$$
We have
\begin{align*}
b_{h,k}&:=\frac{1}{2\pi\sqrt{hk}}\int_{-\infty}^\infty\left|\zeta\left(\frac{1}{2}+it\right)\right|^2\left(\frac{h}{k}\right)^{it}\frac{dt}{\frac{1}{4}+t^2}\\
&=\frac{\log2\pi-\gamma}{2}\left(\frac{1}{h}+\frac{1}{k}\right)+\frac{k-h}{2hk}\log\frac{h}{k}-\frac{\pi}{2hk}\left(V\left(\frac{h}{k}\right)+V\left(\frac{k}{h}\right)\right)\:,\tag{1.3}
\end{align*}
where
$$V\left(\frac{h}{k}\right):=\sum_{m=1}^{k-1}\left\{\frac{mh}{k}\right\}\cot\left(\frac{\pi m h}{k}\right)$$
is Vasyunin's sum (see \cite{VAS}).\\
It can be shown that 
\[
V\left(\frac{h}{k}\right)=-c_0\left(\frac{\bar{h}}{k}\right),\tag{1.4}
\]
where $h\bar{h}\equiv 1(\bmod k)$, with the cotangent sum
$$c_0\left(\frac{h}{k}\right):=\sum_{l=1}^{k-1}\frac{l}{k}\cot\left(\frac{\pi h l}{k}\right)\:,$$
where $h,k\in\mathbb{N}$, $k\geq 2$, $1\leq h\leq k$, $(h,k)=1$ (see \cite{BEC}, \cite{mr}).

Ishibashi \cite{ISH} observed that $c_0$ is related to the value at $s=0$ or $s=1$ 
of the Estermann zeta function by the functional equation of the \textit{imaginary part}. Namely
$$c_0\left(\frac{a}{q}\right)=\frac{1}{2}D_{sin}\left(0,\frac{a}{q}\right)=2q\pi^{-2}D_{\sin}\left(1,\frac{\bar{a}}{q}\right)\:,$$
where for $x\in\mathbb{R}$, $Re(s)>1$, we have:
$$D_{\sin}(s,x):=\sum_{n=1}^\infty\frac{d(n)\sin(2\pi nx)}{n^s}$$
and $a\bar{a}\equiv 1(\bmod q)$.\\
If $x\in\mathbb{R}\setminus\mathbb{Q}$, then de la Bret\`eche and Tenenbaum \cite{bre} showed that the convergence of the above series at $s=1$ is equivalent to the convergence of 
$$\sum_{n\geq 1}(-1)^n\log\frac{v_{n+1}}{v_n}\:,$$
where $u_n/v_n$ denotes the $n$-th partial quotient of $x$.\\
The irrational numbers for which this sum converges are called \textbf{Wilton numbers}.

A basic ingredient in the papers \cite{mr2}, \cite{mr5} have been the representations of 
Balazard, Martin in their papers \cite{balaz1}, \cite{balaz2}, of the function 
$$g(x):=\sum_{l\geq 1}\frac{1-2\{lx\}}{l}\:,$$
involving the Gauss transform from the theory of continued fractions as well as from the paper \cite{Marmi}, by Marmi, Moussa and Yoccoz.\\
These concepts and results also play an important role in the present paper. We shall represent them in the next section.\\
We now describe our result and its relation to the Nyman-Beurling criterion:\\
From (1.2) and (1.3) we obtain
\begin{align*}
&\int_{-\infty}^\infty\left|\zeta\left(\frac{1}{2}+it\right)\right|^2\left|D_N\left(\frac{1}{2}+it\right)\right|^2\frac{dt}{\frac{1}{4}+t^2}\\
&\ \ =\sum_{1\leq h,k\leq N}\mu(h)\mu(k)\left(1-\frac{\log h}{\log N}\right)\left(1-\frac{\log k}{\log N}\right)\\
&\ \ \ \ \ \ \times\left[ \frac{\log 2\pi-\gamma}{2}\left(\frac{1}{h}+\frac{1}{k}\right)+\frac{k-h}{2h\pi}\log\frac{h}{k}-\frac{\pi}{2hk}\left(V\left(\frac{h}{k}\right)+V\left(\frac{k}{h}\right)\right) \right]\:.
\end{align*}
We investigate a partial sum, in which $k$ is kept fixed, namely (up to a constant depending only on $h$):
$$\sum_{h\in I}\mu(h)\left(1-\frac{\log h}{\log N}\right)\frac{1}{h}V\left(\frac{h}{k}\right)\:,$$
which coincides with 
$$\sum_{h\in I}\mu(h)\left(1-\frac{\log h}{\log N}\right)g\left(\frac{h}{k}\right)$$
($I$ is a suitable interval).\\
We prove the following result:
\begin{theorem}\label{main}
Let $0<\delta<D/2$, $k^{2\delta}\leq B\leq k^D$, where $k^{-\delta}\leq \eta\leq 1$. Then there
is a positive constant $\beta$ depending only on $\delta$ and $D$, such that
$$\sum_{Bk\leq h< (1+\eta)Bk}\mu(h)g\left(\frac{h}{k}\right)=O\left((\eta Bk)^{1-\beta}\right)\:.$$
\end{theorem}
\textit{Remark.} This is one of the few cases of a remarkably sharp estimate of a sum 
containing the M\"obius function that is better than the trivial estimate by a factor, which is a
positive power of the number of the terms.
\begin{cor}
Let $\gamma>0$, $\delta>0$, $\epsilon>0$ fixed, $k\geq N^\delta$. Then there
is $\kappa>0$ depending only on $\gamma, \delta, \epsilon$, such that
$$\sum_{1\leq h\leq N}\frac{\mu(h)}{h}\left(1-\frac{\log h}{\log N}\right)V\left(\frac{h}{k}\right)=\sum_{1\leq h\leq k'}\frac{\mu(h)}{h}\left(1-\frac{\log h}{\log N}\right)V\left(\frac{h}{k}\right)+O(k^{1-\kappa})\:,$$
where $k'=k^{1+\epsilon}$.
\end{cor}
\begin{proof}
This follows from Theorem \ref{main} by integration by parts.
\end{proof}
\section{Continued Fractions}
In this section we provide some basic facts related to the theory of continued fractions. We also recall facts and results from the paper \cite{balaz2}.
\begin{definition}\label{def21}
Let $X:=[0,1]\setminus\mathbb{Q}$. For $x\in(0,1)$ we set $\alpha(x):=\{1/x\}$. For $x\in X$ we define recursively
$$\alpha_0(x):=x,\ \alpha_l(x):=\alpha(\alpha_{l-1}(x)),\ \ \text{for $l\in\mathbb{N}$}\:.$$
This definition is also valid for $x\in\mathbb{Q}$ and $l\in\mathbb{N}$, whenever $\alpha_{l-1}(x)\neq 0$. We set
$$a_l(x):=\left[\frac{1}{\alpha_{l-1}(x)}\right]\:.$$
\end{definition}
We have:
\begin{lemma}\label{lem22}
For $x\in X$,
$$x=[0;a_1(x),\ldots, a_l(x),\ldots]$$
the continued fraction expansion of $x$.\\
For $x\in(0,1)\cap \mathbb{Q}$ we have 
$$x=[0;a_1(x),\ldots, a_L(x)]\:,$$
where $L$ is the last $l$, for which $\alpha_{l-1}\neq 0$.\\
We define the partial quotient of $p_l(x), q_l(x)$ by 
$$\frac{p_l(x)}{q_l(x)}:=[0;a_1(x),\ldots, a_l(x)],\ where\ (p_l(x), q_l(x))=1\:.$$
We have
$$p_{l+1}=a_{l+1}p_l+p_{l-1}$$
$$q_{l+1}=a_{l+1}q_l+q_{l-1}\:.$$
\end{lemma}
\begin{proof}
(cf. \cite{hens}, p. 7.)
\end{proof}
\begin{definition}\label{def23}
For $r\in(0,1)\cap\mathbb{Q}$ let
$$r=[0;a_1(r),\ldots, a_L(r)]\:.$$
Then we call $L$ the \textbf{depth} of $r$. 
\end{definition}
\begin{definition}\label{def24}
Let $x\in X$. Then for $l\in\mathbb{N}_0:=\mathbb{N}\cup\{0\}$, we set:
$$\beta_l(x):=\alpha_0(x)\alpha_1(x)\cdots \alpha_l(x)$$
(by convention $\beta_{-1}=1$) and
$$\gamma_l(x):=\beta_{l-1}(x)\log\frac{1}{\alpha_l(x)},\ \text{where}\ l\geq 0,$$
so that $\gamma_0(x)=\log(1/x).$\\
Let $r\in(0,1)\cap\mathbb{Q}$ be a rational number of depth $L$. Then we set:
$$\beta_l(r):=\alpha_0(r)\alpha_1(r)\ldots \alpha_l(r)\:,\ \ \text{for $0\leq l\leq L$ and}$$
$$\gamma_l(r):=\beta_{l-1}(r)\log\frac{1}{\alpha_l(r)}\:,\ \ \text{for $0\leq l\leq L$.}$$
For $x\in X$ we define Wilton's function $\mathcal{W}(x)$ by 
$$\mathcal{W}(x):=\sum_{l\geq 0}(-1)^l\gamma_l(x)\:,\ \ \text{for all $x\in X$,}$$
for which the series is convergent.
\end{definition}
\begin{definition}\label{def25}(definitions from Sec. 4.1 of \cite{balaz2})$ $\\
For $\lambda\geq 0$ let 
$$A(\lambda):=\int_0^\infty\{t\}\{\lambda t\}\frac{dt}{t^2}\:.$$
For $x>0$ let
$$F(x):=\frac{x+1}{2}A(1)-A(x)-\frac{x}{2}\log x\:.$$
For $x\in X$ let
$$G(x):=\sum_{j\geq 0}(-1)^j\beta_{j-1}F(\alpha_j(x))\:.$$
For a rational number $r$ of depth $L$ let
$$G(r):=\sum_{j\leq L}(-1)^j \beta_{j-1}(r)F(\alpha_j(r))\:.$$
\end{definition}
\begin{lemma}\label{lem26}
We have
$$\sup_{\lambda\geq 1}|A(\lambda+h)-A(\lambda)|\leq \frac{1}{2}h\log\left(\frac{1}{h}\right)+O(h)\:,\ \ (0<h\leq 1)\:.$$
\end{lemma}
\begin{proof}
This is Proposition 33 of \cite{balaz2}.
\end{proof}
\begin{definition}\label{def27}
Let 
\begin{eqnarray}
\delta(x):=\left\{ 
  \begin{array}{l l}
   0\: & \quad \text{, if $x\in X$}\vspace{2mm}\\ 
    \frac{(-1)^{L+1}A(1)}{2q}\: & \quad \text{, if $x=p/q\in[0,1]$, $(p,q)=1$, $x$ of depth $L$}\:.\\
  \end{array} \right.
\nonumber
\end{eqnarray}
\end{definition}
\begin{lemma}\label{lem27}
The series $g(x)$ and $\mathcal{W}(x)$ converge for the same values $x\in[0,1]$ and we have 
$$g(x)=\mathcal{W}(x)-2G(x)-2\delta(x)$$
in each point of convergence. 
\end{lemma}
\begin{proof}
This is Proposition 28 of \cite{balaz2}.
\end{proof}
\begin{definition}\label{def28}
For $n\in\mathbb{N}$, $x\in X$ or $x$ a rational number with depth $\geq n$, 
we define 
$$\mathcal{L}(x, n):=\sum_{\nu=0}^n(-1)^\nu(T^\nu l)(x)\:,$$
where $l(x):=\log(1/x)$ and the operator $T\::\: L^p\rightarrow L^p$ is defined
by $$Tf(x):=xf(\alpha(x))\:.$$
\end{definition}
\begin{lemma}\label{lem29}
Let $x\in X$ be a Wilton number or a rational number. Then we have:
$$\mathcal{W}(x)=l(x)-x\mathcal{W}(\alpha(x))\:.$$
\end{lemma}
\begin{proof}
This follows directly from the definition of Wilton's function $\mathcal{W}(x)$.
\end{proof}
\begin{lemma}\label{lem210}
For $n\in\mathbb{N}$, $x\in X$ or $x$ a rational number with depth 
$\geq n+1$ we have 
$$\mathcal{L}(x,n)=\mathcal{W}(x)-(-1)^{n+1}T^{n+1}\mathcal{W}(x)\:.$$
\end{lemma}
\begin{proof}
This follows from Lemma \ref{lem29} by the same computation as in the
proof of Lemma 2.10 in \cite{mr5}, which is valid also for $x$ a rational number with depth $\geq n+1$.
\end{proof}
\begin{lemma}\label{lem211}
For $m\in\mathbb{N}_0$, $x\in X$ or $x$ a rational number of depth $\geq m+1$ we have 
$$\alpha_m(x)\alpha_{m+1}(x)\leq \frac{1}{2}\:.$$
\end{lemma}
\begin{proof}
This is Lemma 2.11 of \cite{mr5}, whose proof is also valid for rational $x$ of depth $\geq m+1$. 
\end{proof}
\begin{definition}\label{def212}
Let $\mathcal{E}$ be a measurable subset of $(0,1)$. The measure $m$ 
is defined by 
$$m(\mathcal{E}):=\frac{1}{\log 2}\int_{\mathcal{E}}\frac{dx}{1+x}\:.$$
\end{definition}
\begin{lemma}\label{lem213}
The measure $m$ is invariant with respect to the map $\alpha$, i.e. 
$$m(\alpha(\mathcal{E}))=m(\mathcal{E})$$
for all measurable subsets $\mathcal{E}\subseteq (0,1)$.
\end{lemma}
\begin{proof}
This result is well-known.
\end{proof}
\begin{lemma}\label{lem213}
Let $n\in \mathbb{N}$. For $f\in L^p$, we have 
$$\int_0^1|T^n f(x)|^p dm(x)\leq g^{(n-1)p}\int_0^1|f(x)|^pdm(x)\:,$$
where 
$$g=\frac{\sqrt{5}-1}{2}<1\:.$$
\end{lemma}
\begin{proof}
For the proof of this result, due to Marmi, Moussa and Yoccoz \cite{Marmi}, see \cite{mr5} Lemma 2.8, (ii).
\end{proof}
\begin{definition}\label{def215}
Let $s\in\mathbb{N}$, $b_0=0$ and $b_1,\ldots, b_s\in\mathbb{N}$. The \textbf{cell} of depth $s$, $\mathcal{C}(b_1,\ldots, b_s)$ is the interval with the endpoints $[0;b_1,\ldots, b_s]$ and \mbox{$[0;b_1,\ldots, b_{s-1},b_s+1]$.}
\end{definition}
\begin{lemma}\label{lem216}
In the interior of the cell $\mathcal{C}(b_1,\ldots, b_s)$ of depth $s$, the 
functions $a_j, p_j, q_j$ are constants for $j\leq s$,
$$a_j(x)=b_j,\ \frac{p_j(x)}{q_j(x)}=[0; b_1,\ldots, b_j],\ (x\in\mathcal{C}(b_1,\ldots, b_j))\:.$$
The endpoints of $\mathcal{C}(b_1,\ldots, b_s)$ are 
\[
\frac{p_s}{q_s}\ \ \text{and}\ \ \frac{p_s+p_{s-1}}{q_s+q_{s-1}}\:.\tag{2.1}
\]
For $x\in X$ and $s\in\mathbb{N}$ there is a unique cell of depth $s$
that contains $x$.\\
Within a cell of depth $s$ we have the derivatives
$$\alpha_s'=(-1)^s(q_s+\alpha_sq_{s-1})^2\:,$$
$$\gamma_s'=(-1)^sq_{s-1}\log\left(\frac{1}{\alpha_s}\right)+(-1)^{s-1}\beta_s\:.$$
\end{lemma}
\begin{proof}
See \cite{balaz2}, sections 2.3 and 2.4.
\end{proof}
\begin{lemma}\label{lem217}
Let $\mathcal{C}(b_1,\ldots, b_s)$ be as in Definition \ref{def215}. Then we have
$$\log\text{meas}(\mathcal{C}(b_1,\ldots, b_s))\leq s\log 2-2\sum_{j=1}^s \log b_j\:.$$
\end{lemma}
\begin{proof}
We compare the measure of cells of depths $j$ and $j+1$. By (2.1) of Lemma 2.16,
we know that $\mathcal{C}(b_1,\ldots, b_j)$ has the endpoints 
$$\frac{p_j}{q_j}\ \ \text{and}\ \ \frac{p_j+p_{j-1}}{q_j+q_{j-1}}$$
and thus the length
$$\left| \frac{p_j}{q_j}-\frac{p_j+p_{j-1}}{q_j+q_{j-1}}\right|=\frac{|p_jq_{j-1}-q_jp_{j-1}|}{q_j(q_j+q_{j-1})}=\frac{1}{q_j(q_j+q_{j-1})}\:.$$
The cell $\mathcal{C}(b_1,\ldots, b_{j+1})$ has the length
$$\frac{1}{q_{j+1}(q_{j+1}+q_j)}\:.$$
Thus
\begin{align*}
\frac{meas (\mathcal{C}(b_1,\ldots, b_{j+1}))}{meas(\mathcal{C}(b_1,\ldots, b_{j}))}&=\frac{q_j(q_j+q_{j-1})}{(b_{j+1}q_j+q_{j-1})((b_{j+1}+1)q_j+q_{j-1})}\\
&\leq \frac{2q_j^2}{b_{j+1}^2q_j^2}=\frac{2}{b^2_{j+1}}\:.\tag{2.2}
\end{align*}
From (2.2) we obtain:
$$meas(\mathcal{C}(b_1,\ldots, b_{s}))=meas (\mathcal{C}(b_1))\prod_{j=1}^{s-1}\frac{meas (\mathcal{C}(b_1,\ldots, b_{j+1}))}{meas(\mathcal{C}(b_1,\ldots, b_{j}))}\leq 2^s\left(\prod_{j=1}^s b_j \right)^{-2}\:,$$
which concludes the proof of Lemma \ref{lem217}.
\end{proof}
\begin{lemma}\label{lem218}
We have $$\log q_s\leq 2\sum_{1\leq j\leq s}\log b_j+s\log 2\:.$$
\end{lemma}
\begin{proof}
This follows from 
$$q_{j+1}=b_{j+1}q_j+q_{j-1}\leq 2b_{j+1}q_j\:.$$
\end{proof}
\begin{lemma}\label{lem219}
Let $C_1>0$ be sufficiently large. Then there exists $C_2>0$, such that for
$s\geq s_0$,
$$meas\{x\in(0,1)\::\: q_s(x)\geq \exp(C_1s)\}\leq \exp(-C_2s)\:.$$
\end{lemma}
\begin{proof}
By partition into cells of depth $s$ we obtain from Lemmas \ref{lem217} and \ref{lem218}
the following:
\begin{align*}
&meas\{x\in(0,1)\::\: q_s(x)\geq \exp(C_1s)\}=\sum_{\substack{\vec{b}=(b_1,\ldots,b_s)\in\mathbb{N}^s\::\\(\sum_{1\leq j\leq s}\log b_j)+s\log 2\geq C_1s}}meas(\mathcal{C}(b_1,\ldots, b_{s}))\\
&\ \ \ \ \ \ \ \leq 2^s\sum_{\substack{\vec{b}\in\mathbb{N}^s\::\\ \sum_{1\le j\leq s}\log b_j\geq (C_1-\log2)s}}b_1^{-2}b_2^{-2}\cdots b_s^{-2}\\
&\ \ \ \ \ \ \ =2^s \sum_{\substack{\vec{b}\in\mathbb{N}^s\::\\ \sum_{1\leq j\leq s}\log b_j\geq (C_1-\log2)s}}\left(\int_{b_1}^{b_1+1}[u_1]^{-2}\:du_1  \right)\cdots \left(\int_{b_s}^{b_s+1}[u_s]^{-2}\:du_s  \right)\\
&\ \ \ \ \ \ \ \leq 4^s \int\displaylimits_{\substack{[1,\infty)^s\::\\ \sum_{1\leq j\leq s}\log u_j\geq (C_1-\log 2)s}} (u_1\cdots u_s)^{-2}\: du_1\cdots du_s\\
&\ \ \ \ \ \ \ =4^s \int\displaylimits_{{\mathbb{R}_+^s\::\:v_1+\cdots+v_s\geq (C_1-\log 2)s}} \exp(-2v_1-\cdots -2v_s)\: dv_1\cdots dv_s\:. \tag{2.3}
\end{align*}
The integral in (2.3) is 
$$2^{-s}Prob(X_1+\cdots+X_s\geq (C_1-\log2)s)\:,$$
where $X_1,\ldots, X_s$ are i.i.d. exponentially distributed random variables with rate 2. The distribution of the sum $X_1+\cdots+X_s$ is the Gamma 
distribution with density 
\[
f(Cs, s, 2)=\frac{2^s(Cs)^{s-1}\exp(-2Cs)}{(s-1)!}\tag{2.4}
\]
Lemma \ref{lem219} now follows from (2.3) and (2.4).
\end{proof}
\section{Vaughan's identity}
We now express the M\"obius function by Vaughan's identity.
\begin{lemma}\label{lem31}
Let $w>1$. For $h\in\mathbb{N}$ we have:
$$\mu(h)=c_1(h)+c_2(h)+c_3(h)\:,$$
where 
$$c_1(h):=\sum_{\substack{\alpha\beta\gamma=h\\\alpha\geq w,\: \beta\geq w}}\mu(\gamma)c_4(\alpha)c_4(\beta)\:,$$
for
$$c_4(\alpha):=-\sum_{\substack{(d_1,d_2)\::\: d_1d_2=\alpha\\ d_1\leq w}}\mu(d_1)\:.$$
\begin{eqnarray}
c_2(h):=\left\{ 
  \begin{array}{l l}
   2\mu(h)\: & \quad \text{, if $h\leq w$}\vspace{2mm}\\ 
    0\: & \quad \text{, if $h>w$}\\
  \end{array} \right.
\nonumber
\end{eqnarray}
and
$$c_3(h):=-\sum_{\substack{\alpha\beta\gamma=h\\\alpha\leq w,\: \beta\leq w}}\mu(\alpha)\mu(\beta)\:.$$
\end{lemma}
\begin{proof}
Cf. \cite{IKW}.
\end{proof}
%
We now obtain 
$$\sum_{Bk\leq h<(1+\eta)Bk}\mu(h)g\left(\frac{h}{k}\right)={\textstyle\sum}_{1}+{\textstyle\sum}_{2}+{\textstyle\sum}_{3}\:,$$
where
\[
{\textstyle\sum}_{i}:=\sum_{Bk\leq h<(1+\eta)Bk}c_i(h)g\left(\frac{h}{k}\right)\:.\tag{3.1}
\]
From the definitions for $c_1$ and $c_4$ in Lemma \ref{lem31} we obtain:
$${\textstyle\sum}_{1}=-\sum_{\substack{s\geq w,\: t\geq w \\ Bk\leq st\gamma<(1+\eta)Bk}}\mu(\gamma)\sum(s,t)\:,$$
where
$$\sum(s,t) := \sum_{\substack{d_1d_2=s \\ d_1\leq w}}\mu(d_1)\sum_{\substack{e_1e_2=t \\ e_1\leq w}}\mu(e_1)g\left(\frac{d_1d_2e_1e_2\gamma}{k}\right)\:.$$
We now choose 
$$w=(Bk)^{2\delta_0},\ \text{for}\ \delta_0>0\ \text{fixed}\:,$$ depending only on $\delta, \eta$ and $D$. The subsequent steps are all valid, if $\delta_0$ is sufficiently small.\\
We partition the sum ${\textstyle\sum}_{1}$ as follows:
$${\textstyle\sum}_{1}={\textstyle\sum}_{1,1}+{\textstyle\sum}_{1,2}$$
with
$${\textstyle\sum}_{1,1}:=-\sum_{\substack{s\geq w,\: t\geq w \\ Bk\leq st\gamma<(1+\eta)Bk\\st\geq (Bk)^{10{\delta_0}}}}\mu(\gamma)\sum(s,t)$$
and
$${\textstyle\sum}_{1,2}:=-\sum_{\substack{s\geq w,\: t\geq w\\ Bk\leq st\gamma<(1+\eta)Bk\\st< (Bk)^{10\delta_0}}}\mu(\gamma)\sum(s,t)\:.$$
A trivial estimate of the sum ${\textstyle\sum}_{2}$ gives:
\[
{\textstyle\sum}_{2}=O\left((Bk)^{2\delta_0} \right)\:.\tag{3.2}
\]
The sums ${\textstyle\sum}_{1,1}$ and ${\textstyle\sum}_{3}$ are estimated by the same method.\\
In the next section we shall give the details of the estimate of ${\textstyle\sum}_{1,1}$, whereas the estimate of ${\textstyle\sum}_{1,2}$ will be given in the last section.
\section{The sums ${\textstyle\sum}_{1,1}$ and ${\textstyle\sum}_{3}$}
We now choose $s_0:=[\delta_1\log k]$, where $\delta_1>0$ depends only on 
$\delta_0$ and $D$. In the sequel we shall always assume that $\delta_1$ is sufficiently small.\\
For the tuplets $(d_1, d_2, e_1, e_2, \gamma)$ appearing in the sum ${\textstyle\sum}_{1,1}$ at least one of the two cases must hold:\\
$$1)\ \ \ \ \ \  \:d_1d_2\gamma e_1\leq (Bk)^{1-\delta_0}$$
$$2)\ \ \ \ \ \ d_1\gamma e_1e_2\leq (Bk)^{1-\delta_0},$$
since otherwise we would have the inequalities 
$$e_2<(Bk)^{2\delta_0}\ \ \text{and}\ \ d_2<(Bk)^{2\delta_0}\:,$$
which in turn would lead to the contradiction
$$st<(Bk)^{8\delta_0}\:.$$
Since the cases 1) and 2) are symmetric, it suffices to treat case 1 only.\\
We write
\[
{\textstyle\sum}_{1,1}={\textstyle\sum}_{1,1}^{(in)} + \sum_{C\in Int}{\textstyle\sum}_{1,1}(C)+{\textstyle\sum}_{1,1}^{(end)}\:, \tag{4.1}
\]
where 
\begin{align*}
&{\textstyle\sum}_{1,1}(C)=-\sum_{\substack{s\geq w,\: t\geq w\\ Ck\leq st\gamma<(C+1)k\\st\geq(Ck)^{10\delta_0}}}\mu(\gamma)\sum(s,t)\:,\\
& \text{$C$ in $\sum_{C\in Int}$ runs over a suitable interval,}\\
& {\textstyle\sum}_{1,1}^{(in)}\ \text{contains $<k$ initial terms, and}\\
& {\textstyle\sum}_{1,1}^{(end)}\ \text{contains $<k$ final terms.}
\end{align*}
We estimate ${\textstyle\sum}_{1,1}^{(in)}$ and ${\textstyle\sum}_{1,1}^{(end)}$ trivially
by 
\[
{\textstyle\sum}_{1,1}^{(in)} \ll k^{1+\epsilon}\tag{4.2}
\]
\[
{\textstyle\sum}_{1,1}^{(end)} \ll k^{1+\epsilon}\:,\ \ (\epsilon>0\ \text{arbitrarily small})\tag{4.3}
\]
and now describe the estimate of the sums ${\textstyle\sum}_{1,1}(C)$.

%
%
%

We set 
$$\vec{d}:=(d_1, d_2, \gamma, e_1),\ \Pi(\vec{d}\:):=d_1d_2\gamma e_1$$ 
and define $e_2^{(0)}$ and $e_2^{(1)}$ as follows:\\
Let 
$$e_2^{(0)}=e_2^{(0)}(\vec{d}\:):=\min\{e_2\::\: e_1e_2\geq w,\: d_1d_2\gamma e_1 e_2\geq Ck\}$$
$$e_2^{(1)}=e_2^{(1)}(\vec{d}\:):=\max\{e_2\::\:  d_1d_2\gamma e_1 e_2\leq (C+1)k\}$$
$$v_{max}:=\max\{v\::\: d_1d_2\gamma e_1(e_2^{(0)}+v)<d_1d_2\gamma e_1(e_2^{(1)}-v)\}\:.$$
We define
$$\sum (\vec{d}\:):=\sum_{Ck\leq \Pi(\vec{d}\:)e_2< (C+1)k}g\left(\frac{\Pi(\vec{d\:}) e_2}{k} \right)$$
and obtain
$$\sum (\vec{d}\:)=\sum_{v\leq v_{max}}\left( g\left(\frac{\Pi(\vec{d\:})( e_2^{(0)}+v)}{k} \right)+g\left(\frac{\Pi(\vec{d\:})( e_2^{(1)}-v)}{k} \right) \right)\:.$$
We now make use of the antisymmetry of the function $g$:
\[
g(C+1-x)=-g(x)\:.\tag{4.4}
\]
We replace the points 
$$\frac{\Pi(\vec{d\:})( e_2^{(1)}-v)}{k}$$
 by the \textit{mirror images} of the points 
 $$\frac{\Pi(\vec{d\:})( e_2^{(0)}+v)}{k}\:,$$
namely by
$$C+1-\frac{\Pi(\vec{d\:})( e_2^{(0)}+v)}{k}$$
and due to (4.1) we obtain
\[\sum (\vec{d}\:)=\sum_{v\leq v_{max}}\left( g\left(\frac{\Pi(\vec{d\:})( e_2^{(1)}-v)}{k} \right)-g\left(C+1-\frac{\Pi(\vec{d\:})( e_2^{(0)}+v)}{k} \right) \right)\:.\tag{4.5}
\]
We have
\[
\left|\frac{\Pi(\vec{d\:})( e_2^{(1)}-v)}{k}-\left(C+1-\frac{\Pi(\vec{d\:})( e_2^{(0)}+v)}{k}\right)\right|\leq \Pi(\vec{d\:}) k^{-1}\:.\tag{4.6}
\]
We now partition the range of summation over the points $\Pi(\vec{d\:})(e_2^{(1)}-v)k^{-1}$ in \textbf{cells of depth $s_0$} and assume that $\delta_1$ is chosen so small that $q_{s_0}(x)\leq k^{\delta_0/10}$ for all cells
$\mathcal{C}$ of depth $s_0$ with the possible exceptions of cells 
$\mathcal{C}$ of depth $s_0$ with total measure $k^{-2\delta_2}$ for some 
$\delta_2>0$.\\
%
%
%
%
%
This can be achieved because of Lemma \ref{lem219}.
Denote the union of these exceptional cells by $\mathscr{C}_{ex}$.\\
For the points 
$$\frac{\Pi(\vec{d\:})( e_2^{(1)}-v)}{k}\in \mathscr{C}_{ex}$$
we estimate the difference in (4.5) by the trivial bound
\[
g\left(\frac{h}{k}\right)=O(\log k)\:.\tag{4.7}
\]
We obtain
\[
{\textstyle\sum}_{ex}\ll k^{1-\delta_2}\:,\tag{4.8}
\]
where ${\textstyle\sum}_{ex}$ is extended over all tuplets $(d_1, d_2, e_1, e_2, \gamma)$ with
$$C\leq \frac{d_1d_2e_1e_2\gamma}{k}<C+1$$
and
$$\left\{ \frac{d_1d_2e_1e_2\gamma}{k}\right\}\in \mathscr{C}_{ex}\:.$$
For the other cells we use that
\[
g(x)=\mathcal{L}(x, s_0)+(-1)^{s_0+1}T^{s_0+1}\mathcal{W}(x)-2G(x)-2\delta(x)\tag{4.9}
\]
and substitute 
$$x^{(0)}:=\frac{\Pi(\vec{d\:})( e_2^{(1)}-v)}{k}\ \ \ \left(\text{resp.}\ \ \ x^{(1)}:=C+1-\frac{\Pi(\vec{d\:})( e_2^{(0)}+v)}{k}\right)\:.$$
This leads to the problem to estimate the differences
$$T^vl(x^{(0)})-T^vl(x^{(1)})$$
and 
$$G(x^{(0)})-G(x^{(1)})\:.$$
This can be done by appeal to Lemmas \ref{lem26}, \ref{lem216} and (4.6). The sum containing the terms $T^{s_0+1}W(x)$ can be estimated using Lemma \ref{lem211}. We obtain
\[
{\textstyle\sum}_{1,1}(C)\ll k^{1-\delta_2}\:.\tag{4.10}
\]
By (4.1), (4.2), (4.3) we get:
\[
{\textstyle\sum}_{1,1}\ll (Bk)^{1-\delta_2}\:.\tag{4.11}
\]
By the same method we derive
\[
{\textstyle\sum}_{3}\ll (Bk)^{1-\delta_2}\:.\tag{4.12}
\]
%
%
%
%
%
%
%
%
%
$$ $$
\section{The sum ${\textstyle\sum}_{1,2}$}
We have
$${\textstyle\sum}_{1,2}=-\sum_{\substack{s\geq w,\: t\geq w\\ Bk\leq st\gamma<Bk(1+\eta)\\st\leq (Bk)^{10\delta_0}}}\mu(\gamma)\sum(s,t)\:,$$
where
$$\sum(s,t)=\sum_{\substack{d_1d_2=s\\ d_1\leq w}}\mu(d_1)\sum_{\substack{e_1e_2=t\\ e_1\leq w}}\mu(e_1)g\left(\frac{d_1d_2e_1e_2\gamma}{k} \right)\:.$$
Therefore
\[
{\textstyle\sum}_{1,2}=-\sum_{Bk\leq u\gamma<Bk(1+\eta)}\mu(\gamma)A(u)g\left(\frac{u\gamma}{k} \right)\:, \tag{5.1}
\]
with
$$A(u):=\sum_{\substack{(s,t)\in\mathbb{N}^2:\\ st=u}}\ \sum_{\substack{d_1d_2=s\\ d_1\leq w}}\mu(d_1)\sum_{\substack{e_1e_2=t\\ e_1\leq w}}\mu(e_1)\:.$$
We observe that
\[
|A(u)|\leq d_4(u)\ll u^{\epsilon}\ \text{for all}\ \epsilon>0\:,\tag{5.2}
\]
where $d_4$ denotes the divisor function sum of order $4$.\\
Let
\begin{align*}
&\mathcal{R}=\{(U,V)\::\: U\geq 0, V\geq 0,\\
&\ \ \ \  \log B+\log k\leq U+V\leq \log B+\log k+\log(1+\eta),\ U\leq 10\delta_0\log k\}\:.\tag{5.3}
\end{align*}
We recursively define the sequence $(\mathcal{R}_l)$ of regions that exhaust $\mathcal{R}$.\\
Let $\mathcal{R}_0$ be the union of squares with sides parallel to the coordinate axes, the 
coordinates of whose vertices are integer multiples of $2^{-0}=1$ and side lengths
$2^{-0}=1$, that are completely contained in $\mathcal{R}$.

The recursion $l\rightarrow l+1$ : Assume that $\mathcal{R}_l$ has been defined. We let
$\mathcal{S}_l$ be the union of squares of sides parallel to the coordinate axes of
side lengths $2^{-(l+1)}$, the coordinates of their vertices are integer multiples of
$2^{-(l+1)}$, whose interiors are contained in the set $\mathcal{R}\setminus \mathcal{R}_l$. We set
$$\mathcal{R}_{l+1}=\mathcal{R}_l\cup \mathcal{S}_l\:.$$
We choose $l_0$ as the minimal $l$, such that 
\[
\max\{u\gamma\::\:(\log u, \log \gamma)\in Q\}-\min\{u\gamma\::\: (\log u, \log \gamma)\in Q\}\leq(\eta Bk)(Bk)^{-2\delta_0}\:,\tag{5.4}
\]
for all squares $Q$ of side lengths $2^{-l_0}$ contained in $\mathcal{R}$.\\
Since the boundaries of $\mathcal{R}$ are straight lines, the trivial bound (4.7) gives
\[
{\textstyle\sum}_{1,2}=-\sum_{(u,\gamma):(\log u,\log \gamma)\in \mathcal{R}_{l_0}}\mu(\gamma)A(u)g\left(\frac{u\gamma}{k}\right)+O\left(\eta Bk(Bk)^{-\delta_0}\right)\:. \tag{5.5}
\]
Let $Q$ be a square of $\mathcal{R}_{l_0}$ and 
$$\text{Exp}(Q):=\{(u,\gamma)\::\:(\log u, \log \gamma)\in Q\}\:.$$
We set
$$\text{Exp}(Q)=:[u_0, u_0+\Delta_u)\times[\gamma_0,\gamma_0+\Delta_\gamma)\:.$$
Because of (5.4) we have
\[
u_0\geq (Bk)^{4\delta_0}\:,\ \ \Delta_u\geq u_0(Bk)^{-2\delta_0}\:,\tag{5.6}
\]
\[
\gamma_0\geq \eta(Bk)^{1-10\delta_0}\:,\ \ \Delta_\gamma\geq \gamma_0(Bk)^{-2\delta_0}>k^{1+10\delta_0}\:,\tag{5.7}
\]
for $\delta_0$ sufficiently small.\\
We now estimate 
\[
{\textstyle\sum}_{1,2}(Q):=-\sum_{(u,\gamma)\in\text{Exp}(Q)}\mu(\gamma)A(u)g\left(\frac{u\gamma}{k}\right)\:. \tag{5.8}
\]
By the Cauchy-Schwarz inequality we have:
\[
\left|{\textstyle\sum}_{1,2}(Q)\right| \leq \left( \sum_{\gamma_0\leq \gamma<\gamma_0+\Delta_\gamma}\mu(\gamma)^2  \right)^{1/2}\left( \sum_{\gamma_0\leq \gamma<\gamma_0+\Delta_\gamma}\left( \sum_{u_0\leq u<u_0+\Delta_u}A(u)g\left(\frac{u\gamma}{k}\right) \right)^2 \right)^{1/2}.  \tag{5.9}
\]
We have
$$ \sum_{\gamma_0\leq \gamma<\gamma_0+\Delta_\gamma}\left( \sum_{u_0\leq u<u_0+\Delta_u}A(u)g\left(\frac{u\gamma}{k}\right) \right)^2
=\sum_{u_0\leq u_1, u_2<u_0+\Delta_u}A(u_1)A(u_2){\textstyle\sum}_{g}(Q,u_1, u_2)\:,$$
where
$${\textstyle\sum}_{g}(Q,u_1, u_2):=\sum_{\gamma_0\leq \gamma<\gamma_0+\Delta_\gamma}g\left(\frac{u_1\gamma}{k}\right)g\left(\frac{u_2\gamma}{k}\right)\:.$$
We now make use of the representation
\[
g(x)=\mathcal{L}(x, s_0)+(-1)^{s_0+1}T^{s_0+1}\mathcal{W}(x)-2G(x)-2\delta(x)\tag{5.10}
\]
which follows from Lemmas \ref{lem27} and \ref{lem210}.\\
We set
$$p(x, u_1, u_2):=(\mathcal{L}(u_1x, s_0)-2G(u_1x))(\mathcal{L}(u_2x, s_0)-2G(u_2x))$$
and replace ${\textstyle\sum}_{g}(Q,u_1, u_2)$ by
\[
{\textstyle\sum}_{g}^{'}(Q,u_1, u_2) :=\sum_{\gamma_0\leq \gamma<\gamma_0+\Delta_\gamma}p(\gamma, u_1, u_2) \tag{5.11}
\]
by removing the terms $T^{s_0+1}\mathcal{W}(\gamma)$, $\delta(\gamma)$ in (5.10).\\
By Lemmas \ref{lem210} and \ref{lem211} it follows that
\[
{\textstyle\sum}_{g} (Q,u_1, u_2)={\textstyle\sum}_{g}^{'}(Q,u_1, u_2)+O(\Delta_\gamma(Bk)^{-\delta_0})\:.\tag{5.12}
\]
We now partition the sum ${\textstyle\sum}_{g}^{'}(Q,u_1, u_2)$ as follows
\[
{\textstyle\sum}_{g}^{'}(Q,u_1, u_2)={\textstyle\sum}_{g,1}^{'}(Q,u_1, u_2)+{\textstyle\sum}_{g, 2}^{'}(Q,u_1, u_2)\:,\tag{5.13}
\]
where in ${\textstyle\sum}_{g,1}^{'}$ the summation is extended over all $\gamma$,
for which
$$\left[\left\{\frac{u_i\gamma}{k}\right\}-\frac{u_i}{k},\ \left\{\frac{u_i\gamma}{k}\right\}+\frac{u_i}{k}\right]\cap \mathscr{C}_{ex}=\emptyset\ \ (i=1, 2)\:,$$
and in ${\textstyle\sum}_{g,2}^{'}$  the summation is extended over the other
values of $\gamma$.\\
By the definition of $\mathscr{C}_{ex}$, for the $\gamma$ appearing in 
${\textstyle\sum}_{g,1}^{'}$  the intervals 
$$\left[\left\{\frac{u_1\gamma}{k}\right\}-\frac{u_1}{2k},\ \left\{\frac{u_1\gamma}{k}\right\}+\frac{u_1}{2k}\right]\ \ \text{resp.}\ \ \left[\left\{\frac{u_2\gamma}{k}\right\}-\frac{1}{2k},\ \left\{\frac{u_2\gamma}{k}\right\}+\frac{1}{2k}\right]$$
belong to cells $\mathcal{C}_1$ resp. $\mathcal{C}_2$ of depth $s_0$ as 
described in Lemma \ref{lem219}.\\
Let $\mathcal{I}$ be an interval of length 1. By Lemma \ref{lem219} we have
$$q_{s_0}(u;\gamma)<\exp(C_1s_0)\ \ \text{for all}\ \gamma=\frac{h}{k}\in\mathcal{I}$$
with the exception of at most
\[
\ll ku_i^{-1}\exp\left(-\frac{C_2s_0}{2}\right)\ \ \text{values}. \tag{5.14}
\]
In ${\textstyle\sum}_{g,1}^{'}$ we replace the terms $p(\gamma, u_1, u_2)$ by their
averages
$$k\int_{\gamma-\frac{1}{2k}}^{\gamma+\frac{1}{2k}} p(x, u_1, u_2)\:dx\:.$$
The intervals 
$$\mathcal{I}_i:=\left[\frac{u_i\gamma}{k}-\frac{u_i}{2k},\ \frac{u_i\gamma}{k}+\frac{u_i}{2k}\right],\ \ (i=1, 2)$$
are contained in some cells $\mathcal{C}_i$ $(i=1, 2)$ of depth $s_0$.\\
We set
\[
{\textstyle\sum}_{g,1}^{'}(Q, u_1, u_2)=I_{g,1}(Q, u_1, u_2)+R_1(Q, u_1, u_2)\:,\tag{5.15}
\]
where
$$I_{g, 1}(Q, u_1, u_2):=k\int_{\gamma_0-\frac{1}{2k}}^{\gamma_0+\Delta_\gamma+\frac{1}{2k}}(\mathcal{L}(u_1x, s_0)-G(u_1x))(\mathcal{L}(u_2x, s_0)-G(u_2x))\: dx\:.$$
By Lemma \ref{lem219} we have
\[
R_1(Q, u_1, u_2)\ll k\Delta_\gamma(Bk)^{-\delta_0}\:.\tag{5.16}
\]
We now write
\[
I_{g,1}(Q, u_1, u_2)=I_{g, 1}'(Q, u_1, u_2)+R_2(Q, u_1, u_2)\:,\tag{5.17}
\]
$$I_{g, 1}'(Q, u_1, u_2):=\int_{\gamma_0-\frac{1}{2k}}^{\gamma_0+\Delta_\gamma+\frac{1}{2k}} g\left( \frac{u_1x}{k} \right)g\left( \frac{u_2x}{k} \right)\: dx$$
and show that $R_2(Q, u_1, u_2)$ is small in average. We have
\begin{align*}
 \tag{5.18}&R_2(Q, u_1, u_2)=k\int_{\gamma_0-\frac{1}{2k}}^{\gamma_0+\Delta_\gamma+\frac{1}{2k}}((-1)^{s_0+1}T^{s_0+1}\mathcal{W}(u_1x)+\delta(u_1x))(\mathcal{L}(u_2x, s_0)+G(u_2x))\:dx\\
&\ \ +k\int_{\gamma_0-\frac{1}{2k}}^{\gamma_0+\Delta_\gamma+\frac{1}{2k}}((-1)^{s_0+1}T^{s_0+1}\mathcal{W}(u_2x)+\delta(u_2x))(\mathcal{L}(u_1x, s_0)+G(u_1x))\:dx\\
&\ \ +k\int_{\gamma_0-\frac{1}{2k}}^{\gamma_0+\Delta_\gamma+\frac{1}{2k}}((-1)^{s_0+1}T^{s_0+1}\mathcal{W}(u_1x)+\delta(u_1x))((-1)^{s_0+1}T^{s_0+1}\mathcal{W}(u_2x)+\delta(u_2x))\: dx\:. 
\end{align*}
We apply Lemma \ref{lem213} for the estimate of 
$$\int_0^1\left| T^{s_0+1}\mathcal{W}(x) \right|^2\: dx$$
and obtain by the Cauchy-Schwarz inequality
\[
\left(\sum_Q{}^{'} R_2(Q, u_1, u_2)^2 \right)^{1/2} \ll\eta(Bk)^{1-\delta_0}\:.\tag{5.19}
\]
In $\sum_Q{}^{'} $ the sum is extended over all squares $Q$ whose union is $\mathcal{R}_{l_0}$.\\
We replace the integrands $g(u_ix/k)$ in (5.17) by finite partial sums of their 
Fourier series 
$$g(u_i x):=h(x, u_i)+r(u_i, x)\:,$$
where
$$h(x, u_i):=\sum_{1\leq n\leq (Bk)^{\delta_4}}\frac{d(n)}{n}\:\sin(2\pi n u_i x)\:,\ \ (i=1, 2)\ (\delta_4>0)\:.$$
By Parseval's equation we obtain
\[
\left(\sum_Q{}^{'} r(u_i, x)^2 \right)^{1/2} \ll\eta(Bk)^{1-\delta_4}\:.\tag{5.20}
\]
We finally also replace ${\textstyle\sum}_{g,2}^{'}(Q, u_1, u_2)$ by
$$\sum_\gamma{}^{'}\int_{\gamma-\frac{1}{2k}}^{\gamma+\frac{1}{2k}}h(x, u_1)h(x, u_2)\: dx\:,$$
where $\gamma$ is extended over the same values as in formula (5.13).\\
From (5.9), (5.11), (5. 12), (5.13), (5.14), (5.15), (5.16), (5.19) and (5.20) we obtain:
\begin{align*}
\tag{5.21} {\textstyle\sum}_{1,2}(Q)&\ll k^\epsilon \Delta_\gamma^{1/2}\sum_{1\leq n_1, n_2\leq \eta(Bk)^{\delta_1}}\frac{d(n_1)}{n_1}\frac{d(n_2)}{n_2} \\
&\times\sum_{u_0\leq u_1, u_2<u_0+\Delta_u}\int_{\gamma_0}^{\gamma_0+\Delta_\gamma}\sin(2\pi n_1u_1v)\sin(2\pi n_2u_2v)\: dv+R(Q)\:,
\end{align*}
where 
$$\sum_Q{}^{'}|R(Q)|\ll \eta(Bk)^{1-\delta_5}\ \ (\delta_5>0)\:.$$
From the identity
$$\int_0^1\sin(2\pi n_1u_1v)\sin(2\pi n_2u_2v)\: dv = \left\{ 
  \begin{array}{l l}
    1/2\:, & \quad \text{if $n_1u_1=n_2u_2$}\vspace{2mm}\\ 
   0\:, & \quad \text{otherwise}\:,\\
  \end{array} \right.
\nonumber$$
we obtain
\[
 {\textstyle\sum}_{1,2} \ll \eta(Bk)^{1-\delta_6}\ \ (\delta_6>0)\:.  \tag{5.22}
\]
$$ $$
\section{Conclusion}
Theorem \ref{main} now follows from (3.1), (3.2), (4.11), (4.12) and (5.22).

\end{document}